\documentclass[12pt]{article}

\usepackage{enumerate}
\usepackage{amsmath}
\usepackage{amssymb,latexsym}
\usepackage{amsthm}
\usepackage{color}
\usepackage{graphicx}
\usepackage{hyperref}
\usepackage{amscd}

\title{Codes with few weights arising from linear sets}
\author{Vito Napolitano and Ferdinando Zullo}
\date{}

\newcommand{\cC}{{\mathcal C}}

\newcommand{\cH}{{\mathcal H}}
\newcommand{\cD}{{\mathcal D}}

\newcommand{\cL}{{\mathcal L}}

\newcommand{\C}{{\mathcal C}}
\newcommand{\F}{{\mathbb F}}

\newcommand{\la}{\langle}
\newcommand{\ra}{\rangle}

\newtheorem{theorem}{Theorem}[section]

\newtheorem{lemma}[theorem]{Lemma}
\newtheorem{corollary}[theorem]{Corollary}

\newtheorem{result}[theorem]{Result}

\newtheorem{remark}[theorem]{Remark}

\DeclareMathOperator{\PG}{{PG}}

\begin{document}
\maketitle

\begin{abstract}
In this article we present a class of codes with few weights  arising from special type of linear sets.
We explicitly show the weights of such codes, their weight enumerator and possible choices for their generator matrices.
In particular, our construction yields also to linear codes with three weights and, in some cases, to almost MDS codes.
The interest for these codes relies on their applications to authentication codes and secret schemes, and their connections with further objects such as association schemes and graphs.
\end{abstract}



\smallskip

\thanks{This research was partially supported by the Italian National Group for Algebraic and Geometric Structures and their Applications (GNSAGA - INdAM). The authors were also supported by the project ''VALERE: VAnviteLli pEr la RicErca" of the University of Campania ''Luigi Vanvitelli''.}

\section{Introduction}

Let $q=p^h$, with $p$ prime and $h \geq 1$.
A $q$-\emph{ary linear code} $\C$ is any $\F_q$-subspace of $\F_{q}^m$.
If $\C$ has dimension $r$, we say that $\C$ is an $[m,r]_q$-\emph{code}.
A \emph{generator matrix} $G$ for a linear code $\C$ is a $(r \times m)$-matrix over $\F_q$ whose rows form a basis of $\C$, i.e.
\[ \C=\{\mathbf{x}G \colon \mathbf{x}\in \F_q^r\}. \]
We may consider $\F_q^m$ as a metric space endowed with the \emph{Hamming distance} $d_H$, i.e. $d_H(\mathbf{x},\mathbf{y})$ is the
number of entries in which $\mathbf{x}$ and $\mathbf{y}$ differ, with $\mathbf{x},\mathbf{y} \in \F_q^m$.
The \emph{Hamming weight} $w_H(\mathbf{c})$ of a codeword $\mathbf{c}\in \C$ is the number of nonzero entries of $\mathbf{c}$, or equivalently $w_H(\mathbf{c})=d_H(\mathbf{c},\mathbf{0})$.
Recall that the minimum weight of a linear code coincides with its minimum distance.
Let denote by $A_i^H$ the number of codewords of $\C$ with Hamming weight $i$.
The \emph{(Hamming) weight enumerator} is defined as the following polynomial
\[ 1+A_1^H z+\ldots+A_m^H z^m. \]
This polynomial gives a good deal of information about the code and it is an important invariant, and has been calculated for few families of codes. Also, it is used in the probability theory involved with different ways of decoding.
A $\ell$-\emph{weight code} $\C$ is an $[m,r]_q$-code having $\ell$ nonzero weights $w_1<\ldots<w_{\ell}$, i.e. if the sequence $(A_1^H,\ldots,A_m^H)$ have exactly $\ell$ nonzero entries. If $\ell\leq r$ we say that the code has \emph{few weights}.
Much of the focus on linear codes to date has been on codes with few weights, especially on two and three-weight codes, for their applications in secret sharing \cite{DingDing}, authentication codes \cite{aa} and their connections with association schemes \cite{CK1986,CG1984} and with graphs \cite{ShiSo}.

\smallskip

Let $n,r,h$ be positive integers such that $h\leq r-1$ and $h+1 \mid rn$, and let $q$ be a prime power. Denote by $\theta_i=\frac{q^{i+1}-1}{q-1}$ for any $i\geq 0$.
In this paper we deal with a geometric construction of linear $[\theta_{\frac{rn}{h+1}},r]_{q^n}$-codes with at most $h+1$ weights, determining their weight enumerator and the possible distribution of their weights, see Section \ref{sec:construction}.
In particular, for $h=r-1$ this construction yields  $r$-weight codes connected with a special class of maximum rank distance codes, see Section \ref{se:h=r-1}. When $h=2$ we prove that our construction yields three-weight codes, cf. Section \ref{sec:h=2}, and examples of almost MDS codes in even characteristic, c.f. Remark \ref{rk:almostMDS}.
Finally, in Section \ref{sec:gen}, we present suitable choices for generator matrices when $h=2$ and $h=r-1$.

\section{Some known linear codes with few weights}\label{sec:codesfeww}

In the literature the known constructions for linear codes with few weights arise either from linear algebraic tools, based on the properties of the trace function, or from a geometric point of view.

\smallskip

In \cite{DingLuo,DingN}, the authors define a class of linear codes as follows. Let $\mathrm{Tr}_{q/p}$ denote the trace function from $\F_q$ to $\F_p$, i.e. $\mathrm{Tr}_{q/p}(x)=\sum_{i=0}^{h-1}x^{p^i}$. Let $D=\{d_1,\ldots,d_m\}\subseteq \F_q$ and let
\[ \C_D=\{ (\mathrm{Tr}_{q/p}(xd_1),\ldots,\mathrm{Tr}_{q/p}(xd_m)) \colon x \in \F_q \}. \]
Then $\C_D$ is a linear code of length $m$ over $\F_p$ and $D$ is called its \emph{defining set}.
Most of the known constructions for codes with few weights come from $\C_D$ by selecting  appropriately the defining set $D$.
For instance, in \cite{DingDing} the authors choose
\[ D=\{ x\in \F_q^* \colon \mathrm{Tr}_{q/p}(x^2)=0 \}. \]
Then
\[ |D|=\left\{ \begin{array}{ll }p^{h-1}-1 & h>1\,\,\text{odd}\\ p^{h-1}-(-1)^{\left( \frac{p-1}2 \right)^2\frac{m}2}(p-1)p^{\frac{m-2}2}-1 & h\geq 2\,\,\text{even} \end{array}  \right. \]
and $\C_D$ has length $|D|$ and either two or three different weights according to the parity of $h$.

\smallskip

For the geometric type construction we need the following notion.
A {\em projective $[n,  r, d]$--system } is a finite (multi-)set $\mathcal M$  of points of $\mathrm{PG}(r - 1, q)$, not
all of which lie in a hyperplane, where $n = \vert {\mathcal M}\vert$ , and
\[n - d = \max\{\vert {\mathcal M}\cap  H\vert \, :  H \subset \mathrm{PG}(r - 1, q), \, \dim H = r - 2\},\]
see  \cite{AL2019}.
Note that the cardinalities above are counted with multiplicities in the case of a multiset.
By defining the $(r \times n)$-matrix $G$ by taking as columns the coordinates of points of $\mathcal{M}$, then $G$ is the generator matrix of a linear $[n,r]_q$-code $\C_{\mathcal{M}}$.

Conversely, let $G$ be an $(r \times n)$-generator matrix of a nondegenerate $[n, r]_q$-linear code $\C$, that is a code without coordinates being zero for every codeword. The (multi-)set of one dimensional subspaces of $\F_q^r$ spanned by the columns of $G$, may be considered as a (multi-)set $\mathcal M$ of points of $\mathrm{PG}(r- 1, q)$.
So, there is a one to one correspondence between nondegenerate linear codes and projective systems. Furthermore, for any non-zero vector ${\bf v} = (v_1, v_2, \ldots , v_r)$ in $\F_q^r$, we have that the projective hyperplane
\[v_1x_1 + v_2x_2 + \cdots + v_rx_r = 0\]
contains $\vert \mathcal {M}\vert  - w$  points of $\mathcal M$ if and only if the codeword ${\bf v}G$ has weight $w$.  Thus, there exists a linear $[n, r]_q$--code with minimum distance $d$ if and only if there exists a projective $[n,  r, d]$-system.
So, the number of distinct weights of $\C$ corresponds to the distinct sizes of the intersections of ${\mathcal M}$ with all the hyperplanes, that is with the intersection numbers of ${\mathcal M}$ with respect to\ the hyperplanes.
See also \cite{TVN}.

In particular, we are interested in three-weight codes, i.e. those arising from (multi-)set with three intersection numbers with respect to hyperplanes.
Despite of the number of the algebraic constructions for three-weight codes, see e.g. \cite{DingDing,DingDing2,Ding3,ShiSo,Wu}, there are few known geometric constructions \cite{Durante}.
A general geometric construction for three-weight codes has been given in \cite{AgugliaGiuzzi}.
Let $a \in \F_{q^2}^*$, $b \in \F_{q^2}\setminus\F_q$
and $\mathcal{B}(a,b)$ be the affine algebraic set of $\mathrm{AG}(r-1,q^2)$ of equation
\[ x_{r-1}^q-x_{r-1}+a^q(x_1^{2q}+\ldots+x_{r-2}^{2q})-a(x_1^{2}+\ldots+x_{r-2}^{2})=(b^q-b)(x_1^{q+1}+\ldots+x_{r-2}^{q+1}). \]
Let $P_{\infty}$ be the point at infinity of the non-degenerate Hermitian variety of $\mathrm{AG}(r-1,q^2)$ with affine equation
\[ x_{r-1}^q-x_{r-1}= (b^q-b)(x_1^{q+1}+\ldots+x_{r-2}^{q+1}).\]
Let $r$ be a odd positive integer and let $j$ be a positive integer greater than one. Providing some conditions on $a,b$ and $q$ (cf. \cite[Section 5]{AgugliaGiuzzi}), the multiset $\overline{\mathcal{B}}(a,b)$, consisting of the points of $\mathcal{B}(a,b)$ and $j$ times the point $P_{\infty}$, defines a linear $[q^{2r-3}+j,r]_{q^2}$-code $\C_{\overline{\mathcal{B}}(a,b)}$ with weights
\begin{itemize}
    \item $q^{2r-3}$;
    \item $q^{2r-3}-q^{2r-5}$;
    \item $q^{2r-3}-q^{2r-5}+q^{r-3}+j$;
    \item $q^{2r-3}-q^{2r-5}+q^{r-3}-q^{r-2}+j$.
\end{itemize}
When $j=q^{r-2}-q^{r-3}$ or $j=q^{2r-5}-q^{r-3}$, $\C_{\overline{\mathcal{B}}(a,b)}$ is a three-weight code. Note that when $\C_{\overline{\mathcal{B}}(a,b)}$ is a three-weight code, it is also $q$-\emph{divisible}, i.e. all the nonzero weights are divisible by $q$.

\section{Scattered linear sets}\label{sec:lienarsets}

Let $V$ be an $r$-dimensional $\F_{q^n}$-vector space. A point set $L$ of $\Lambda=\PG(V,\F_{q^n})\allowbreak=\PG(r-1,q^n)$ is said to be an \emph{$\F_q$-linear set} of $\Lambda$ of rank $k$ if it is defined by the non-zero vectors of a $k$-dimensional $\F_q$-vector subspace $U$ of $V$, i.e.
\[L=L_U:=\{\la {\bf u} \ra_{\mathbb{F}_{q^n}} \colon {\bf u}\in U\setminus \{{\bf 0} \}\}.\]
We will also denote its rank by $\mathrm{rk}(L_U)$.

Let $\Omega=\PG(W,\F_{q^n})$ be a subspace of $\Lambda$ and let $L_U$ be an $\F_q$-linear set of $\Lambda$, then $\Omega \cap L_U=L_{W\cap U}$, and if $\dim_{\F_q} (W\cap U)=i$, i.e. if the $\F_q$-linear set $\Omega \cap L_U=L_{W\cap U}$ has rank $i$, we say that $\Omega$ has \emph{weight} $i$ in $L_U$, and we write $w_{L_U}(\Omega)=i$.
Note that if $\Omega$ has dimension $s$ and $L_U$ has rank $k$, then $0 \leq w_{L_U}(\Omega) \leq \min\{k,n(s+1)\}$.
In particular, a point $P$ belongs to an $\F_q$-linear set $L_U$ if and only if $w_{L_U}(P)\geq 1$.

For an $\F_q$-linear set $L_U$ of rank $k$ in $\Lambda=\PG(r-1,q^n)=\PG(V,\F_{q^n})$ the bound
\begin{equation}\label{cardlinearsets}
|L_U| \leq \theta_{k-1}=\frac{q^k-1}{q-1},
\end{equation}
holds true. Hence, $L_U$ is called \emph{scattered} if it achieves the bound \eqref{cardlinearsets},
or equivalently if all of its points have weight one. 

A scattered $\F_q$-linear set $L_U$ of $\Lambda$ with highest possible rank is a {\it maximum scattered $\F_q$--linear set} of $\Lambda$.

More recently, extending the definition given in \cite{Lunardon2017} and in \cite{ShVdV}, the family of $h$-scattered linear sets has been introduced in \cite{CsMPZ}.
An $\F_q$-linear set $L_U$ of $\Lambda$ is said $h$-\emph{scattered linear set} if
\begin{itemize}
  \item $\langle L_U \rangle= \Lambda$;
  \item for each $(h-1)$-subspace $\Omega$ of $\PG(r-1,q^n)$ we have that
\[ w_{L_{U}}(\Omega) \leq h. \]
\end{itemize}
When $h=r-1$ and $\dim_{\F_q}U=n$, we obtain the scattered linear sets with respect to hyperplanes introduced in \cite{Lunardon2017} and in \cite{ShVdV}.

\smallskip

In this paper we will deal with $h$-scattered linear sets of maximum rank, showing that they provide families of point set with few intersection numbers with respect to hyperplanes, yielding codes (equipped with Hamming distance) with few weights.
To this aim, let us recall some properties of $h$-scattered linear sets proved in \cite{CsMPZ}.
In \cite[Theorem 2.3]{CsMPZ}, it has been proved a bound on the rank of a $h$-scattered $\F_q$-linear set.
More precisely, if $L_U$ is an $h$-scattered $\F_q$-linear set in $\Lambda=\PG(r-1,q^n)$ which is not a subgeometry, then
\begin{equation}\label{eq:boundrank}
\mathrm{rk}(L_U) \leq \frac{rn}{h+1},
\end{equation}
which generalizes the bound given by Blokhuis and Lavrauw in \cite{BL2000} for the classical scattered linear sets.
Constructions of $h$-scattered $\F_q$-linear sets attaining the bound \eqref{eq:boundrank} have been presented in \cite[Theorems 2.7 and 3.6]{CsMPZ} for $h>1$ and in \cite[Section 4]{NPZZ}, \cite[Section 3]{Lunardon2017} and in \cite[Corollary 4.4]{ShVdV} for $h=r-1$.
Therefore, $h$-scattered $\F_q$-linear sets attaining \eqref{eq:boundrank} are called \emph{maximum $h$-scattered} $\F_q$-linear sets.
From now on, we will consider only maximum $h$-scattered $\F_q$-linear sets.
In \cite[Theorem 2.8 and Section 5]{CsMPZ}, the authors also show that such linear sets have few possible intersection numbers with respect to hyperplanes.
Indeed, if $L_U$ is a maximum $h$-scattered $\F_q$-linear set in $\Lambda=\PG(r-1,q^n)$ then for each hyperplane $\mathcal{H}=\PG(r-2,q^n)$ of $\Lambda$ we have that
\[ \frac{rn}{(h+1)}-n\leq w_{L_U}(\mathcal{H}) \leq \frac{rn}{(h+1)}-n+h. \]
Since $L_U$ is scattered, if $\mathcal{H}$ is a hyperplane of $\Lambda$ with $w_{L_U}(\mathcal{H})=i$, then
\[ |L_U \cap \mathcal{H}| =\theta_{i-1}=\frac{q^i-1}{q-1},\]
hence for any hyperplane $\mathcal{H}$ we have that the possible $h+1$ intersection numbers with respect to hyperplanes are
\begin{equation}\label{eq:intnumb}
\theta_{\frac{rn}{(h+1)}-n-1}, \ldots, \theta_{\frac{rn}{(h+1)}-n+h-1}.
\end{equation}
Remark that $\theta_{-1}=0$.
For $h=1$, in \cite{BL2000}, the authors proved that there exist hyperplanes intersecting the linear set in $\theta_{\frac{rn}{2}-n-1}$ and $\theta_{\frac{rn}{2}-n}$ points, i.e. maximum scattered linear sets are two character sets.
When $h>1$ it is not known whether for all the intersection numbers in \eqref{eq:intnumb} there exist at least one hyperplane intersecting the linear set in such a number of points.
We will solve this problem for the $h=r-1$ and $h=2$ cases, see Sections \ref{se:h=r-1} and \ref{sec:h=2}.
Since the intersection numbers with respect to hyperplanes for a maximum $h$-scattered linear set are at most $h+1$, they provide examples of codes with few weights, as we will see in the next sections.

\section{Constructions}\label{sec:construction}

In this section we exhibit codes (equipped with Hamming metric) with few weights arising from maximum $h$-scattered linear sets following the geometric construction presented in Section \ref{sec:codesfeww}, and the papers \cite{CG1984,CK1986} and \cite[Section 2.6]{LavrauwThesis}.

\smallskip

Now, let $L_U$ be a maximum $h$-scattered linear set in $\Lambda=\PG(r-1,q^n)$.
Denote by $N$ its number of points, i.e. $N=\frac{q^{\frac{rn}{h+1}}-1}{q-1}$, and let
\[L_U=\{ \langle(g_{1i},\ldots,g_{ri})\rangle_{\F_{q^n}} \colon i\in \{1,\ldots,N\} \}.\]
Consider $G$ as the $(r\times N)$-matrix with the points of $L_U$ as columns.
Since $\langle L_U \rangle=\Lambda$, the rank of $G$ is $r$, so $G$ is the generator matrix of an $\F_{q^n}$-linear code $\cC_{L_U}$ of dimension $r$ and length $N$.

\medskip

Consider $\mathbf{c}=(c_1,\ldots,c_N)$ a codeword of $\cC_{L_U}$, then there exists $(a_0,\ldots,a_{r-1})\in \F_{q^n}^r$ such that
\[ c_j=\sum_{i=1}^{r} a_{i-1} g_{ij} \]
for each $j \in \{1,\ldots,N\}$. Suppose that $c_j=0$, then $a_0 g_{1j} + \cdots + a_{r-1} g_{rj}=0$.
This happens if and only if the point $\langle (g_{1j},\ldots,g_{rj}) \rangle_{\F_{q^n}}$ belongs to the hyperplane with equation
$a_0 x_0+\cdots+a_{r-1}x_{r-1}=0$.
Therefore, denoting by $\cH \colon a_0 x_0+\cdots+a_{r-1}x_{r-1}=0$, since $L_U$ is scattered then the Hamming weight of $c$ is $N-\theta_{i-1}$  if $w_{L_U}(\cH)=i$.
Hence, by \eqref{eq:intnumb}, we have the following.

\begin{theorem}\label{th:(h+1)weights}
If $L_U$ is a maximum $h$-scattered linear set in $\Lambda$, then $\C_{L_U}$ is a code with at most $h+1$ weights.
In particular, the weights of $\C_{L_U}$ are
\[ N-\theta_{i-1}, \]
where $i$ runs over all the possible weights of a hyperplane with respect to $L_U$.
\end{theorem}

In the next sections, we will determine the weights and the weight enumerators for $h=2$ and $h=r-1$.

\section{Weights for $h=r-1$}\label{se:h=r-1}

As already remarked, to find the weights of the  aforementioned codes is equivalent to finding all sizes of intersections with respect to the hyperplanes of a fixed maximum $h$-scattered linear set, i.e. its intersection numbers with respect to the hyperplanes.
In Section \ref{sec:lienarsets}, we mentioned that the possible intersection numbers of a maximum $h$-scattered linear set with respect to hyperplane are among a collection of $h+1$ values, but it is not known whether for each of these values we may find at least one hyperplane intersecting the linear set in such a number of points.
Here, we shall show that for $h=r-1$, that all the aforementioned intersection numbers are nonzero and hence we get an $r$-weight code.

\subsection{Connections with MRD-codes}\label{subsec:MRD}

Before considering the $h=r-1$ case, i.e. scattered linear sets with respect to hyperplanes, we briefly point out some basic properties of MRD-codes and their connection with such linear sets.
A \emph{rank metric (or RM) code} $\cC$ of $\F_q^{n\times m}$, $n \leq m$, can be considered as a subset of $\mathrm{Hom}_{\F_q}(U,V)$, where $\dim_{\F_{q}} U = m$ and $\dim_{\F_{q}} V = n$,
with \emph{rank distance} defined as $d(f,g):=\mathrm{rk}(f-g)$, i.e. $\dim_{\F_q}\mathrm{Im}(f-g)$.
The minimum distance of $\cC$ is $d:=\min \{d(f,g)\colon f,g\in \cC, f\neq g\}$.

\begin{result}\cite{Delsarte}
\label{Dels}
	If $\cC$ is a rank metric code of $\F_q^{n\times m}$, $n \leq m$, with minimum distance $d$, then
\begin{equation}
\label{MRDbound}
	|\cC|\leq q^{m(n-d+1)}.
\end{equation}
\end{result}

Rank metric codes for which \eqref{MRDbound} holds with equality are called \emph{maximum rank distance (or MRD) codes}.

\medskip

We will only consider $\F_q$-linear MRD-codes of $\F_q^{n\times n}$, i.e. those which can be identified with $\F_q$-subspaces of $\mathrm{End}_{\F_q}(\F_{q^n})$.
Since $\mathrm{End}_{\F_q}(\F_{q^n})$ is isomorphic to the ring $\cL_{n,q}$ of $q$-polynomials over $\F_{q^n}$  modulo $x^{q^n}-x$, i.e.
\[\cL_{n,q}\simeq\left\{ \sum_{i=0}^{n-1} a_ix^{q^i}, \,\,\, a_i \in \F_{q^n} \right\},\]
with addition and composition as operations, we will consider $\cC$ as an $\F_q$-subspace of $\cL_{n,q}$, see e.g. \cite{wl}.
We will need the following property on the number $A_i$ of matrices in $\C$ having rank $i$.

\begin{result}\label{res:weights}\cite[Lemma 2.1]{LTZ2}
If $\cC$ is an MRD-code of $\F_q^{n\times n}$ with minimum distance $d$, then  $A_i>0$ for each $i$ such that $d\leq i\leq n$.
\end{result}

More generally, the weight distribution of an MRD-code was precisely determined by Delsarte in \cite{Delsarte} (and later by Gabidulin in \cite{Gabidulin}).

\begin{result}\label{weightdistribution}
Let $\cC$ be an MRD-code in $\F_q^{n\times n}$ with minimum distance $d$. Then
\begin{equation}\label{eq:Ad+l}
A_{d+\ell}={n \brack d+\ell}_q \sum_{t=0}^\ell (-1)^{t-l}{\ell+d \brack \ell-t}_q q^{\binom{\ell-t}{2}}(q^{n(t+1)}-1),
\end{equation}
for $\ell \in \{0,1,\ldots,n-d\}$.
\end{result}

Given two $\F_q$-linear RM-codes, $\cC_1$ and $\cC_2$, they are equivalent if and only if there exist two invertible $\F_q$-linear map $\varphi_1$, $\varphi_2\in \cL_{n,q}$ and $\rho\in \mathrm{Aut}(\F_q)$ such that
\[ \varphi_1\circ f^\rho \circ \varphi_2 \in \cC_2 \text{ for all }f\in \cC_1,\]
where $\circ$ stands for the composition of maps modulo $x^{q^n}-x$ and $f^\rho(x)= \sum a_i^\rho x^{q^i}$ for $f(x)=\sum a_i x^{q^i}$.
For a rank metric code $\cC$, its left and right idealisers can be written as:
\[L(\cC)= \{ \varphi \in \cL_{n,q}\colon \varphi \circ f \in \cC \text{ for all }f\in \cC \},\]
\[R(\cC)= \{ \varphi \in \cL_{n,q}\colon f \circ \varphi \in \cC \text{ for all }f\in \cC \},\]
see \cite{LN2016} and \cite{LTZ2}.

\smallskip

The relation between scattered linear sets with respect to hyperplanes and MRD-codes has been pointed out in \cite{Lunardon2017} and in \cite{ShVdV}, see also \cite[Section 4.1]{CsMPZ}.
More precisely, $\cC$ is an $\F_q$-linear MRD-code of $\cL_{n,q}$ with minimum distance $n-r+1$ and with left-idealiser isomorphic to $\F_{q^n}$ if and only if $\cC$ is equivalent to
\[\la f_1(x),\ldots,f_r(x)\ra_{\F_{q^n}}\]
for some $f_1,f_2,\ldots,f_r \in \cL_{n,q}$, which happens if and only if
\[L_U=\{\langle(f_1(x),\ldots,f_r(x))\rangle_{\F_{q^n}} \colon x\in \F_{q^n}^*\}\]
is a scattered $\F_q$-linear set with respect to hyperplanes of rank $n$ in $\Lambda=\PG(\F_{q^n}^r,\F_{q^n})=\PG(r-1,q^n)$.

\subsection{Scattered $\F_q$-linear sets with respect to hyperplanes}

Let $L_U$ be a scattered linear set with respect to\ the hyperplanes of $\Lambda=\PG(r-1,q^n)$ with $\mathrm{rk}(L_U)=n$, i.e. $L_U$ is a maximum $(r-1)$-scattered linear set.
Since $\mathrm{rk}(L_U)=n$, then there exist $f_1(x),\ldots,f_r(x) \in \cL_{n,q}$ such that
\[ L_U=\{\langle (f_1(x),\ldots,f_r(x)) \rangle_{\F_{q^n}}\colon x\in \F_{q^n}^*\}. \]
For each hyperplane $\cH=\PG(W,\F_{q^n})=\PG(r-2,q^n)$ of $\Lambda=\PG(r-1,q^n)$, by \eqref{eq:intnumb} we get that
\[ w_{L_U}(\cH)\in \{0,\ldots,r-1\}, \]
i.e.
\begin{equation}\label{eq:intnumbr-1}
|\cH \cap L_U| \in \{ 0,1,q+1,\ldots,q^{r-2}+\cdots+q+1 \}.
\end{equation}

\medskip

We show that for each $i \in \{0,\ldots,r-1\}$ there exists at least one hyperplane $\cH$ such that $w_{L_U}(\cH)=i$.

\smallskip

\begin{remark}\label{weight}
Note that if $\cH: a_0x_0+\cdots+a_{r-1}x_{r-1}=0$ is a hyperplane of $\Lambda$, then
\[w_{L_U}(\cH)=n-\mathrm{rk}(a_0f_1(x)+\cdots+a_{r-1}f_r(x))=\dim_{\F_q}\ker(a_0f_1(x)+\cdots+a_{r-1}f_r(x)),\]
c.f. \cite[Proposition 4.2]{ShVdV}
\end{remark}

\smallskip

We prove that $L_U$ is an $r$-character set with respect to hyperplanes, i.e. it has $r$ distinct intersection numbers with respect to the hyperplanes.

\begin{theorem}\label{th:r-char}
If $L_U$ is a scattered linear set with respect to hyperplanes in $\Lambda=\PG(r-1,q^n)$, then it is an $r$-character set with respect to hyperplanes with intersection numbers given in \eqref{eq:intnumbr-1}.
\end{theorem}
\begin{proof}
Consider $i \in \{0,\ldots,r-1\}$, since $L_U$ is a scattered linear set with respect to hyperplanes, then by Subsection \ref{subsec:MRD} the rank metric code $\cD=\langle f_1(x),\ldots,f_r(x) \rangle_{\F_{q^n}}$ is an MRD-code. In particular, by Result \ref{res:weights}, there exists in $\cD$ at least one $q$-polynomial, say $f(x)=a_0x+\ldots+a_{r-1}x^{q^{r-1}}$, whose image has dimension $n-i$ over $\F_q$, i.e. its  kernel has dimension $i$. Let $\cH\colon a_0x_0+\ldots+a_{r-1}x_{r-1}=0$.
By Remark \ref{weight}, it follows that the hyperplane $\cH$ has weight $i$ with respect to $L_U$, i.e. it meets $L_U$ in $\theta_{i-1}=q^{i-1}+\ldots+q+1$ points.
\end{proof}

\begin{corollary}\label{cor:Ai}
The number of hyperplanes in $\Lambda$ having weight $i$ w.r.t $L_U$ (i.e. the number of hyperplanes meeting $L_U$ in $\theta_{i-1}$ points) is $A_{n-i}$, where $A_{j}$'s are defined as in \eqref{eq:Ad+l} and $i\in \{0,\ldots,r-1\}$.
\end{corollary}

Therefore, as a consequence Theorem \ref{th:(h+1)weights} and Theorem \ref{th:r-char} we have the following result.

\begin{corollary}
If $L_U$ is a maximum $(r-1)$-scattered in $\Lambda$, then $\C_{L_U}$ is an $r$-weight code with weights
\[ w_i=N-\theta_{i-1}, \]
with $i \in \{0,\ldots,r-1\}$. Its weight enumerator is
\[ 1+\sum_{i=0}^{r-1} A_{n-i} z^{w_i}, \]
where $A_j$'s are given in \eqref{eq:Ad+l}.
\end{corollary}

\section{Weights for $h=2$}\label{sec:h=2}

Let $L_U$ be a maximum $2$-scattered linear set in $\Lambda=\PG(r-1,q^n)$, in particular $\mathrm{rk}(L_U)=\frac{rn}3$.
Since $L_U$ is scattered we also have that $|L_U|=\theta_{\frac{rn}3-1}$.
By \eqref{eq:intnumb}, we have that for each hyperplane $\cH=\PG(r-1,q^n)$ it follows that
\[ w_{L_U}(\cH) \in \left\{ \frac{(r-3)n}3, \frac{(r-3)n}3+1, \frac{(r-3)n}3+2 \right\}, \]
i.e.
\[ |\cH \cap L_U| \in \left\{ \theta_{\frac{(r-3)n}3-1}, \theta_{\frac{(r-3)n}3}, \theta_{\frac{(r-3)n}3+1} \right\}. \]
We are going to show that $L_U$ is a three-character set with respect to hyperplanes.

\begin{remark}
In this section, $n\geq 3$, since we are considering linear sets generating the all space with rank $\frac{rn}3$, i.e.
\[ r\leq \mathrm{rk}(L_U) \leq \frac{rn}3, \]
which cannot be satisfied when $n\leq 2$.
\end{remark}

Denote by $t_i$ the number of hyperplanes of $\Lambda$ with weight $\frac{(r-3)n}3+i$ with respect to $L_U$, for $i \in \{0,1,2\}$.
Since for each line $\ell$ of $\Lambda$ we have that $|L_U \cap \ell| \leq q+1$ and counting the hyperplanes, point-hyperplane pairs $(A,\cH)$ with $A\in \cH \cap L_U$ and point-point-hyperplane triples $(A,B,\cH)$ with $A\neq B$ and $A,B \in \cH \cap L_U$, we get the following equations in the variables $t_0$, $t_1$ and $t_2$,

\begin{equation}\label{eq:sist}
\begin{array}{lll}
\vspace{0.5cm} t_0+t_1+t_2=\frac{q^{nr}-1}{q^n-1};\\ \vspace{0.5cm}
\theta_{\frac{(r-3)n}3-1} t_0+\theta_{\frac{(r-3)n}3} t_1+\theta_{\frac{(r-3)n}3+1} t_2=\theta_{\frac{rn}3-1} \frac{q^{n(r-1)}-1}{q^n-1};\\
\theta_{\frac{(r-3)n}3-1}\left( \theta_{\frac{(r-3)n}3-1} -1\right) t_0+\theta_{\frac{(r-3)n}3}\left( \theta_{\frac{(r-3)n}3} -1\right) t_1+\theta_{\frac{(r-3)n}3+1} \left( \theta_{\frac{(r-3)n}3+1} -1\right) t_2=\\=\theta_{\frac{rn}3-1}\left( \theta_{\frac{rn}3-1}-1 \right) \frac{q^{n(r-2)}-1}{q^n-1}.
\end{array}
\end{equation}

\medskip
\newpage

By the previous system we get that

\[t_0=\frac{\left(q^{\frac{r n}{3}}-1\right)}{(q-1)^2 (q+1) \left(q^n-1\right)}t_0^*; \]
\begin{equation}\label{eq:t0,t1,t2}t_1=\frac{\left(q^{\frac{r n}{3}}-1\right) \left(q^{\frac{r n}{3}+2}-q^{\frac{r n}{3}+1}-q^n+q^2\right)}{(q-1)^2 q};\end{equation}
\[ t_2=\frac{\left(q^n-q\right) \left(q^{\frac{r n}{3}}-1\right)}{(q-1)^2 q (q+1)}, \]
where
\[ t_0^*= q^{\frac{2 r n}{3}+3}-q^{\frac{2 r n}{3}+2}-q^{\frac{2 r n}{3}+1}+q^{\frac{2 r n}{3}}-q^{\frac{r n}{3}+n+2}+q^{\frac{r n}{3}+3}-q^{\frac{r n}{3}+1}+q^{\frac{1}{3} (r+3) n}+q^{2 n}-q^{n+2}-q^{n+1}+q^3. \]

\begin{remark}\label{rem:r=3}
For $r=3$, maximum $2$-scattered linear sets correspond to scattered with respect to lines (hyperplanes) in $\Lambda=\PG(2,q^n)$.
Hence, by Theorem \ref{th:r-char} we have that $t_0, t_1$ and $t_2$ are positive and by Corollary \ref{cor:Ai} it follows that
\[ t_i=A_{n-i}, \,\,\,\text{for}\,\,\, i \in\{0,1,2\}. \]
\end{remark}

\begin{lemma}
The values $t_0$, $t_1$ and $t_2$ are nonzero for each $r\geq 3$, $q\geq 2$ and $n\geq 3$.
\end{lemma}
\begin{proof}
By Remark \ref{rem:r=3}, we may assume $r\geq 4$.
Since $q\geq 2$, $n\geq 3$ and $r\geq 4$, it is clear that $t_2>0$.
First, note that for any $q$
\begin{equation}\label{eq:1} q^k>\frac{q^k-1}{q-1}=q^{k-1}+q^{k-2}+\ldots+q+1. \end{equation}
Let $m=\frac{rn}3$. Then $m>n\geq 3$ and $t_1>0$ follows from
\[ q^{m+2}-q^{m+1}-q^n+q^2>q^{m+2}-q^{m+1}-q^m>0,\]
by \eqref{eq:1}.

Now, we show that $t_0>0$.
This follows from
\[ t_0^*=q^{2m+3}-q^{2m+2}-q^{2m+1}-q^{m+n+2}+q^{2m}+q^{m+3}+q^{m+n}-q^{m+1}+q^{2n}-q^{n+2}-q^{n+1}+q^3= \]
\[ =q^{m+n+2}(q^{m-n+1}-q^{m-n}-q^{m-n-1}-1)+q^m(q^m+q^3+q^n-q)+q^n(q^n-q^2-q)+q^3
\]
\[ >q^{m+n+2}(q^{m-n+1}-q^{m-n}-q^{m-n-1}-1)+q^m(q^m+q^3+q^n-q)+q^n(q^n-q^2-q)>0,
\]
by \eqref{eq:1} and since $q\geq 2$ and $n\geq 3$.
\end{proof}

Therefore, we have the following.

\begin{theorem}\label{th:3-char}
The linear set $L_U$ is a three-character set with respect to the hyperplanes with intersection numbers $\theta_{\frac{(r-3)n}3-1}, \theta_{\frac{(r-3)n}3}$ and $\theta_{\frac{(r-3)n}3+1}$.
Also, the number of hyperplanes intersecting $L_U$ in $\theta_{\frac{(r-3)n}3+i-1}$ points is given by $t_i$, with $i \in \{0,1,2\}$.
\end{theorem}

Whereas, using Theorems \ref{th:(h+1)weights} and \ref{th:3-char}, we have the following.

\begin{corollary}
If $L_U$ is a maximum $2$-scattered linear set in $\Lambda$, then $\C_{L_U}$ is a three-weight code with weights
\[ w_0:=N-\theta_{\frac{(r-3)n}3-1}, \]
\[ w_1:=N-\theta_{\frac{(r-3)n}3}, \]
and
\[ w_2:=N-\theta_{\frac{(r-3)n}3+1}. \]
Its weight enumerator is
\[ 1+t_0 z^{w_0}+t_1z^{w_1}+t_2z^{w_2}, \]
where $t_i$'s are defined as in \eqref{eq:t0,t1,t2}.
\end{corollary}

\subsection{Maximum scattered linear sets with respect to lines of the plane}

When $h=2$ and $r=3$, by \eqref{eq:t0,t1,t2}, the number of lines $t_i$ with weight $i$ with respect to $L_U$ are

\[ t_0= \frac{3 q^{2 n}-2 q^{n+1}-q^{n+2}+q^{n+3}-q^{2 n+1}-2 q^{2 n+2}+q^{2 n+3}+q^3}{(q-1)^2 (q+1)}; \]

\[ t_1= \frac{\left(q^n-1\right) \left(-q^{n+1}+q^{n+2}-q^n+q^2\right)}{(q-1)^2 q}; \]

\[ t_2= \frac{\left(q^n-1\right) \left(q^n-q\right)}{(q-1)^2 q (q+1)}. \]

By the aforementioned construction, we get a three-weight code $\mathcal{C}_{L_U}$ of dimension $3$ and length $N$, with weights

\[ w_1=N=\frac{q^n-1}{q-1}; \]

\[ w_2=N-1=q \frac{q^{n-1}-1}{q-1}; \]

\[ w_3=N-q-1=q^2 \frac{q^{n-2}-1}{q-1}. \]

Furthermore, the weight enumerator is

\[ 1 + t_2 z^{N-q-1} + t_1 z^{N-1} + t_0 z^N. \]

\smallskip

\begin{remark}\label{rk:almostMDS}
Since the minimum distance is $N-q-1$, the Singleton bound for such a code is
\[ N-q-1 \leq N - 2. \]
In particular, the defect in the Singleton bound is $q-1$. For $q=2$ we get an \emph{almost MDS code}, i.e. a code for which the minimum distance is the difference of its length and its dimension.
De Boer in \cite{DeBoer} introduced those codes for the first time and very recently in \cite{Mehta} they have been used in secret sharing schemes.
\end{remark}

\subsection{Comparison with known families of three-weight codes}

Most of the known three-weight codes beloging to the first family mentioned in Section \ref{sec:codesfeww} have the property that their lenght divides the order of the base field minus one, see \cite[Section VI]{DingDing}.
Our construction provides codes for which this property does not hold.

When $r>3$ the codes we construct are $q$-divisible, whereas for $r=3$ they are not, contrary to those presented in \cite{AgugliaGiuzzi}.
Furthermore, the codes introduced in \cite{AgugliaGiuzzi} are defined over $\F_{q^2}$, whereas our construction yield codes defined over $\F_{q^n}$, for any $n\geq 3$.

\section{Generator matrices}\label{sec:gen}

The aim of this section is to exhibit generator matrices of the aforementioned linear codes.

\smallskip

We start by considering an $\F_q$-linear set $L_U$ which is scattered with respect to hyperplanes of rank $n$ in $\PG(r-1,q^n)$, which, as seen before, may be written as
\[ L_U=\{ \langle (f_1(x),\ldots,f_r(x)) \rangle_{\F_{q^n}} \colon x \in \F_{q^n} \}, \]
with $f_1(x),\ldots,f_r(x) \in \mathcal{L}_{n,q}$.
Therefore, a matrix whose columns correspond to the following set of vectors
\[ \{(f_1(x),\ldots,f_r(x)) \colon x \in \F_{q^n}^*\} \]
is a generator matrix for the code $\mathcal{C}_{L_U}$.

In the following table, we resume the known choices of the polynomials $f_1,\ldots,f_r$ defining scattered $\F_q$-linear sets with respect to the hyperplanes in $\PG(r-1,q^n)$.

\begin{table}[htp]
\[
  \begin{array}{ |c|c|c|c|c| }
\hline
n & r & (f_1(x),\ldots,f_r(x)) & \mbox{conditions} & \mbox{references} \\ \hline
& & (x,x^{q^s},\ldots,x^{q^{s(r-1)}}) & \gcd(s,n)=1 & \cite{Delsarte,Gabidulin,kshevetskiy_new_2005} \\ \hline
 & & (x^{q^s},\ldots,x^{q^{s(r-2)}},x+\delta x^{q^{s(r-1)}}) & \begin{array}{cc} \gcd(s,n)=1,\\ \mathrm{N}_{q^n/q}(\delta)\neq (-1)^{nr}\end{array} & \cite{Sheekey2016,LTZ}\\ \hline
6 & 4 & (x^q, x^{q^2}, x^{q^4},x-\delta^{q^5} x^{q^{3}})  &  \begin{array}{cc} q>4 \\ \text{certain choices of} \, \delta \end{array} & \cite{CMPZ,PZ} \\ \hline
6 & 4 & (x^q,x^{q^3},x-x^{q^2},x^{q^4}-\delta x) & \begin{array}{cccc}q \hspace{0.1cm} \text{odd} \\ \delta^2+\delta =1 \end{array}
 & \cite{CsMZ2018,MMZ} \\ \hline
6 & 4 & \begin{array}{cc} (h^{q^2-1}x^q+h^{q-1}x^{q^2},x^{q^3},\\ x^q-h^{q-1}x^{q^4},x^q-h^{q-1}x^{q^5}) \end{array} & \begin{array}{cccc}q \hspace{0.1cm} \text{odd} \\ h^{q^3+1}=-1 \end{array}
& \cite{BZZ,ZZ} \\ \hline
7 & 3 & (x,x^q,x^{q^3}) & \begin{array}{cc} q \hspace{0.1cm} \text{odd},\\ \gcd(s,7)=1\end{array} & \cite{CsMPZh} \\ \hline
7 & 4 & (x,x^{q^{2s}},x^{q^{3s}},x^{q^{4s}}) & \begin{array}{cc} q \hspace{0.1cm} \text{odd},\\ \gcd(s,n)=1\end{array} & \cite{CsMPZh} \\ \hline
8 & 3 & (x,x^q,x^{q^3}) & \begin{array}{cc} q \equiv 1 \pmod{3},\\ \gcd(s,8)=1\end{array}  & \cite{CsMPZh} \\ \hline
8 & 5 & (x,x^{q^{2s}},x^{q^{3s}},x^{q^{4s}},x^{{5s}}) & \begin{array}{cc} q \equiv 1 \pmod{3},\\ \gcd(s,8)=1 \end{array} & \cite{CsMPZh} \\ \hline
8 & 6 & (x^q,x^{q^2},x^{q^3},x^{q^5},x^{q^6},x-\delta x^{q^4}) & \begin{array}{cc} q\,\text{odd},\\ \delta^2=-1\end{array} & \cite{CMPZ} \\ \hline
\end{array}
\]
\caption{Possible choices for $f_1,\ldots,f_r$}
\label{kMRD}
\end{table}

\newpage

Hence, we have the following result.

\begin{theorem}
Let $q$ be a prime power, $r,n$ be two positive integers with $r\geq 3$, let $(f_1(x),\ldots,f_r(x))$ be as in Table \ref{kMRD} and $N=\frac{q^n-1}{q-1}$. If $G$ is a $(r\times N)$-matrix whose set of columns is
\[ \{(f_1(x),\ldots,f_r(x)) \colon x \in \F_{q^n}^*\}, \]
then $G$ is the generator matrix of a $r$-weight linear $[N,r]_{q^n}$-code.
\end{theorem}

For maximum $2$-scattered $\F_q$-linear sets we have less examples.
A general construction for $h$-scattered linear sets is given by the following result.

\begin{theorem}\label{th:directsum}{\rm \cite[Theorem 2.5]{CsMPZ}}
Let $V=V_1\oplus \ldots \oplus V_t$, let
$L_{U_i}$ a $h_i$-scattered $\F_q$-linear set in $\PG(V_i,\mathbb{F}_{q^n})$, with $i\in \{1,\ldots,t\}$ and let
\[U=U_1\oplus \ldots \oplus U_t.\]
The linear set $L_U$ is $h$-scattered in $\Lambda=\PG(V,\mathbb{F}_{q^n})$, with $h=\min\{h_1,\ldots,h_t\}$.
\end{theorem}

Hence, suppose that $3 \mid r$. So, $r=3t$ and let $\Lambda=\PG(\F_{q^n}^{3t},\F_{q^n})$.
Let $\pi=\PG(W_i,\F_{q^n})$ be the projective subspace of $\Lambda$ defined by
\[ \pi_i\colon x_j=0, \,\,\,\,\, j \in \{ 1,\ldots,3t \}\setminus\{3i-2,3i-1,3i\} \]
for $i \in \{1,\ldots,t\}$.

Thus we have
\[ \Lambda=\langle \pi_1,\ldots,\pi_t\rangle, \]
and
\[ \F_{q^n}^{3t}=W_1\oplus \ldots \oplus W_t. \]

For each $i \in \{1,\ldots,t\}$, choose
\[ (f_1^i(x),f_2^i(x),f_3^i(x)) \]
as in Table \ref{kMRD} with $r=3$.
Then consider $U_i$ the set of vectors $(a_1,\ldots,a_{3t})$
\[ a_j=\left\{ \begin{array}{llll} f_1^i(x) & \text{if}\, j=i\\
f_2^i(x) & \text{if}\, j=i+1\\
f_3^i(x) & \text{if}\, j=i+2\\
0 & \text{otherwise}.
\end{array} \right. \]
It follows that $U_i$ is an $\F_q$-subspace and $L_{U_i}$ is a $2$-scattered $\F_q$-linear set in $\pi_i$.
Therefore, by Theorem \ref{th:directsum} we have that
\[ U=U_1\oplus \ldots \oplus U_t= \]
\[ \{ (f_1^1(x_1),f_2^1(x_1),f_3^1(x_1),\ldots,f_1^t(x_t),f_2^t(x_t),f_3^t(x_t)) \colon x_i \in \F_{q^n} \} \]
defines a maximum $2$-scattered $\F_q$-linear set in $\Lambda$.
Hence, a $(3\times \theta_{tn-1})$-matrix $G$ over $\F_{q^n}$ whose columns correspond to the following set of vectors
\[ \{(f_1^1(x_1),f_2^1(x_1),f_3^1(x_1),\ldots,f_1^t(x_t),f_2^t(x_t),f_3^t(x_t))  \colon x_i \in \F_{q^n}\}\setminus\{\mathbf{0}\} \]
is a generator matrix for the code $\mathcal{C}_{L_U}$.

Therefore, we have the following.

\begin{theorem}
Let $q$ be a prime power, $t$ be a positive integer $t\geq 2$, $N=\theta_{tn-1}=\frac{q^{nt}-1}{q-1}$ and $(f_1^j(x),f_2^j(x),f_3^j(x))$ as in Table \ref{kMRD} for each $j \in \{1,\ldots,t\}$.
If $G$ is a $(3\times N)$-matrix $G$ over $\F_{q^n}$ whose columns correspond to the following set of vectors
\[ \{(f_1^1(x_1),f_2^1(x_1),f_3^1(x_1),\ldots,f_1^t(x_t),f_2^t(x_t),f_3^t(x_t))  \colon x_i \in \F_{q^n}\}\setminus\{\mathbf{0}\}, \]
then $G$ is the generator matrix of a three-weight linear $[N,3]_{q^n}$-code.
\end{theorem}

\bigskip

\noindent Vito Napolitano and Ferdinando Zullo\\
Dipartimento di Matematica e Fisica,\\
Universit\`a degli Studi della Campania ``Luigi Vanvitelli'',\\
Viale Lincoln, 5\\
I--\,81100 Caserta, Italy\\
{{\em \{vito.napolitano,ferdinando.zullo\}@unicampania.it}}

\end{document}